\RequirePackage{fix-cm}
\documentclass[smallextended]{svjour3}
\smartqed

\usepackage{amsmath,amssymb,amsfonts}
\usepackage{bm}
\usepackage{graphicx}
\usepackage{color}
\usepackage{enumerate}
\usepackage{algorithm}
\usepackage{algpseudocode}
\usepackage{float}
\usepackage[colorlinks,linkcolor=blue,citecolor=blue,urlcolor=blue]{hyperref}

\newtheorem{prop}[theorem]{Proposition}

\newtheorem{rem}[theorem]{Remark}
\newtheorem{eg}[theorem]{Example}
\newtheorem{lem}[theorem]{Lemma}
\newtheorem{cor}[theorem]{Corollary}

\newcommand{\RR}{\mathbb{R}}

\newcommand{\lmo}{\operatorname{LMO}}
\newcommand{\proj}{\operatorname{Proj}}
\newcommand{\diam}{\operatorname{Diam}}
\newcommand{\scal}[2]{{({{#1},{#2}})}}

\title{Robustness of the Frank-Wolfe Method under Inexact Oracles and the Cost of Linear Minimization}

\titlerunning{Frank-Wolfe with inexact/stochastic gradients; LMO vs projection}

\author{Tao Hu    
}

\institute{Tao Hu \at
              School of Mathematics and Statistics, Xi'an Jiaotong University, Xi'an, Shaanxi 710049, P.R. China \\ \email{hu\_tao@stu.xjtu.edu.cn} 
}

\date{}

\begin{document}
\maketitle

\begin{abstract}
We investigate the robustness of the Frank-Wolfe method when gradients are computed inexactly and examine the relative computational cost of the linear minimization oracle (LMO) versus projection. For smooth nonconvex functions, we establish a convergence guarantee of order $\mathcal{O}(1/\sqrt{k}+\delta)$ for Frank-Wolfe with a $\delta$--oracle. Our results strengthen previous analyses for convex objectives and show that the oracle errors do not accumulate asymptotically. We further prove that approximate projections cannot be computationally cheaper than accurate LMOs, thus extending to the case of inexact projections. These findings reinforce the robustness and efficiency of the Frank-Wolfe framework.
\end{abstract}

\section{Introduction}\label{sec:intro}
The Frank-Wolfe method is a projection-free method for constrained optimization.
At each iteration, Frank-Wolfe solves a linear minimization subproblem instead of a projection, which can be advantageous on structured domains such as simplices or spectrahedra. Classical analyses guarantee an $\mathcal{O}(1/k)$ convergence rate for convex objectives \cite{Jaggi2013,FreundGrigas2013} and $\mathcal{O}(1/\sqrt{k})$ for nonconvex ones \cite{lacostejulien2016convergenceratefrankwolfenonconvex}.

Despite its popularity, the robustness of Frank-Wolfe under inexact gradient information remains incompletely understood. In deterministic optimization, the $\delta$--oracle model \cite{Devolder2014} is used to bound the gradient error. We revisit this framework and derive new nonaccumulation bounds for both convex and nonconvex settings.

Let $Q\subset\RR^d$ be a compact convex set and $f:Q\to\RR$ be the objective function. Denote $\|\cdot\|$ to be the $l^2$ norm. We consider the minimization problem:\begin{equation*}
\label{poi3}\begin{array}{rl}\min\limits_{ x} & f( x) \\
\mathrm{s.t.} &  x \in Q \ . \end{array}
\end{equation*} 

The Frank-Wolfe method, which has a linear minimization oracle as its basic module, is an effective way to address this problem. At the iteration point $ x_k\in Q$, the Frank-Wolfe method solves the linear minimization subproblem
\begin{equation*}\label{eq:lmo}
  \tilde{ x}_k\in \arg\min_{ x\in Q}\big\{f( x_k)+\nabla f( x_k)^\top( x- x_k)\big\}
\end{equation*}
and updates with $ x_{k+1}=(1-\overline\alpha_k) x_k+\overline\alpha_k \tilde{ x}_k$ where $\overline\alpha_k\in[0,1)$.

In the convergence analysis of the Frank–Wolfe method, the following auxiliary sequences are frequently used and will also appear in our proofs:
\begin{equation}\label{dos}
\beta_k = \frac{1}{\prod_{j=0}^{k-1} (1 - \overline{\alpha}_j)}, 
\qquad 
\alpha_k = \frac{\beta_k \overline{\alpha}_k}{1 - \overline{\alpha}_k},
\qquad k \ge 1.
\end{equation}
Here, $\{\overline{\alpha}_k\}_{k=0}^{+\infty}$ denotes the sequence of step sizes used in our algorithm.
We follow the conventions $\prod_{j=0}^{-1} (\cdot) := 1$ and $\sum_{i=0}^{-1} (\cdot) := 0$.

We denote the Frank-Wolfe gap at the point $ x \in Q$ as
\[G( x) = \sup_{ y \in Q} \nabla f( x)^T( x -  y).\]

Besides the convergence guarantee, robustness and the efficiency of the Linear Minimization Oracle are also important aspects of the Frank-Wolfe method.

The performance of Frank-Wolfe under inexact gradient is a very interesting problem. With unbiased gradients and bounded variance (or sub-Gaussian tails), Stochastic Frank-Wolfe variants achieve a Frank-Wolfe gap of $\mathcal{O}(\varepsilon)$ with $\mathcal{O}(1/\varepsilon^4)$ gradient evaluations, and variance reduction accelerates finite-sum problems and can achieve the same Frank-Wolfe gap with $\mathcal{O}(1/\varepsilon^3)$ gradient evaluations~\cite{ReddiSFW2016,Goldfarb2017,LuFreund2018,Locatello2019,Zhang2020,sfyraki2025lionsmuonsoptimizationstochastic}. For heavy-tailed noise, Stochastic Frank-Wolfe with clipping or robust estimation achieves high-probability guarantees~\cite{Tang2022,SfyrakiWang2025}.  

In the deterministic setting, there is situation where the gradient error is bounded by $\delta/D$ but can be arbitrarily chosen along the training trajectory. This case is often relaxed and referred to as obtaining a $\delta$-oracle: 
\begin{equation}\label{eq:delta-oracle}\left|(g_{\delta}( x) - \nabla f( x))^T( x - y)\right| \leq \delta,\, \forall \; y \in Q.\end{equation}

If we use an increasingly accurate gradient along the iterates, like if
\[\left|(g_{\delta}( x_k) - \nabla f( x_k))^T( x_k - y)\right| \leq \frac{1}{k + 1}\delta L D^2,\, \forall \; y \in Q,\]
then
\[G( x_k) \leq \frac{27LD^2}{4(k + 2)}(1 + \delta).\]
For the more common scenario, where the error does not decrease, Freund and Grigas prove an $\mathcal{O}(1/k + \delta)$ convergence of the Frank-Wolfe gap\cite{FreundGrigas2013}, and we show an $\mathcal{O}(1/\sqrt{k} + \delta)$ convergence for nonconvex functions in this paper. 

Sometimes we consider objective functions that are convex but non-smooth or the case when the gradients are computed at shifted points \cite{DevolderGlineurNesterov2013}. Those functions may not obtain a gradient, but they can be equipped with a $(\delta, L)$ oracle \cite{Devolder2014}:\[
0 \leq f( x) - (f_{\delta, L}(y) + g_{\delta, L}(y)^T( x - y)) \leq \frac{L}{2}\| x -y\|^2 + \delta,\, \forall \;  x, y \in Q.
\]
Since the first bound of the Frank-Wolfe gap under $(\delta, L)-$ oracle, which is $\mathcal{O}(1/k + k\delta)$, has been proposed\cite{FreundGrigas2013}, it has been an open problem for more than ten years whether the final guarantee of the Frank-Wolfe gap is optimal theoretically.

Besides robustness, the ease of computing the Linear Minimization Oracle is widely considered another major advantage of the Frank–Wolfe method, 
which makes it more prevalent than proximal gradient methods. 
However, this belief is currently supported mainly by intuition and set-specific comparisons 
\cite{CombettesPokutta2021,PokuttaIntro2024}. 
Additionally, we address a recent question posed by Woodstock~\cite{Woodstock2025} on whether linear minimization oracles are inherently cheaper than projection operators. We show that even coarse, approximate projections cannot outperform accurate LMOs in computational complexity.

Our main contributions are as follows:
\begin{enumerate}[(i)]
\item \textbf{Nonconvex Frank-Wolfe with a $\delta$-oracle.} We prove that for $L$-smooth nonconvex objectives, Frank-Wolfe with a $\delta$-oracle achieves
\[
\min_{0\le k\le K} G( x^k)\ \le\ \sqrt{\frac{2C\,(f( x^0)-f^*)}{K+1}}\ +\ 2\delta.
\]
\item \textbf{Projection vs.\ LMO.} We prove that a $K$-approximate projection at a scaled point $-\lambda x$ produces an $\varepsilon$-accurate LMO at $x$ with $\varepsilon=\mathcal{O}((K+D^2)/\lambda)$, establishing that approximate projections can not be uniformly easier than LMOs.
\end{enumerate}

\section{Frank-Wolfe with an Inexact Oracle}\label{sec:delta}

Let $Q\subset\RR^d$ be compact and convex, and $f:Q\to\RR$ be convex with $L$-Lipschitz gradient. The Frank-Wolfe update in Algorithm~\ref{alg:fw-delta} uses the $\delta$-oracle $g_\delta$ defined in \eqref{eq:delta-oracle}.
\begin{algorithm}[h]
\caption{Frank-Wolfe with a $\delta$-oracle}
\label{alg:fw-delta}
\begin{algorithmic}[1]
\State Initialize $x_0\in Q$.
\For{$k=0,1,2,\ldots$}
\State Query $g_\delta(x_k)$.
\State Solve $\tilde x_k=\arg\min_{x\in Q}g_\delta(x_k)^\top(x-x_k)$.
\State Set $x_{k+1}=x_k+\overline\alpha_k(\tilde x_k-x_k)$, $\overline\alpha_k\in[0,1)$.
\EndFor
\end{algorithmic}
\end{algorithm}

\begin{lem}\label{lem:wolfe-transfer}
Under \eqref{eq:delta-oracle}, for any $x_k\in Q$,
\[
f^\star \;\ge\; f(x_k) \;+\; \min_{x\in Q} g_\delta(x_k)^\top(x-x_k) \;-\; \delta .
\]
\end{lem}
\begin{proof}
By convexity, $f(x)\ge f(x_k)+\nabla f(x_k)^\top(x-x_k)$ for any $x\in Q$. From \eqref{eq:delta-oracle},
$\nabla f(x_k)^\top(x-x_k)\ge g_\delta(x_k)^\top(x-x_k)-\delta$. Therefore,
\[f(x) \geq f(x_k) + g_\delta(x_k)^T (x - x_k) - \delta.\]
Taking $\min_{x\in Q}$ on both sides yields the claim.
\end{proof}

We also recall a subproblem-level accuracy transfer.
\begin{prop}[{\cite[Prop.\ 5.1]{FreundGrigas2013}}]\label{prop:FG-subproblem}
Fix $\overline x\in Q$ and $\delta\ge 0$. If $\tilde x\in\arg\min_{x\in Q} g_\delta(\overline x)^\top x$, then
\[
\nabla f(\overline x)^\top\tilde x \;\le\; \min_{ x\in Q} \nabla f(\overline x)^\top x \;+\; 2\delta .
\]
\end{prop}

The convergence theorem of Frank-Wolfe with a $\delta$- oracle on convex objectives is given by Freund and Grigas as follows (one can actually show that the result of Freund and Grigas actually applies to the widest range of step-sizes): 

\begin{theorem}[Nonaccumulation under $\delta$-oracle, convex case\cite{FreundGrigas2013}]\label{thm:delta}
Let $Q$ be compact convex with diameter $D$, and $f$ be convex with $L$-Lipschitz gradient on $Q$. Let $g_\delta$ satisfy \eqref{eq:delta-oracle}. For the Frank-Wolfe iterates of Algorithm~\ref{alg:fw-delta} with stepsize satisfying $\sum_{k=0}^{+\infty}\overline\alpha_k=\infty$, $\sum_{k=0}^{+\infty}\overline\alpha_k^2<\infty$ and $\overline\alpha_k\downarrow 0$, then 
\begin{equation}\label{eq:one-step}
f( x_{k+1}) - f^*\le (1-\overline\alpha_k)\big(f( x_k) - f^*\big)+2\overline\alpha_k\delta+\tfrac{1}{2}LD^2\overline\alpha_k^2,
\end{equation}
and hence $\displaystyle \limsup_{k\to\infty}(f( x_k) - f^\star)\le 2\delta$.
\end{theorem}

\begin{eg}[Tightness up to constants]\label{ex:tight}
Let $Q=[-1,1]$, $f( x)=\tfrac12 x^2$ (convex, $L=1$, $D=2$). Define a $\delta$-oracle by $g_\delta( x)=\nabla f( x)-\tfrac{\delta}{D}\,\operatorname{sign}( x)$. Frank-Wolfe with $\overline\alpha_k=2/(k+2)$ converges to a neighborhood whose size is proportional to $\delta$.
\end{eg}

\section{Nonconvex Frank-Wolfe with an Inexact Oracle}\label{sec:nonconvex-delta}

We now consider \emph{nonconvex} minimization over a compact convex set $S\subset\RR^d$:
\[
\min_{x\in S} f(x),
\]
where $f$ is differentiable and has $L$-Lipschitz gradient on $S$. Denote $D:=\diam(S)$ and set
\[
C \;\triangleq\; \max\{\,L D^2,\; G D\,\}\quad\text{with}\quad G:=\sup_{x\in S}\|\nabla f(x)\|<\infty .
\]
The Frank-Wolfe gap at $x$ is
\[
G(x)\;\triangleq\;\max_{s\in S}\,\langle \nabla f(x),\,x-s\rangle .
\]
We assume access to a $\delta$-oracle for the gradient, i.e., for every $x\in S$ there exists $g_\delta(x)$ such that
\begin{equation}\label{eq:dir-delta}
\big|\langle \nabla f(x)-g_\delta(x),\,s-x\rangle\big|\ \le\ \delta\quad\forall\,s\in S .
\end{equation}
Define the \emph{approximate Frank-Wolfe gap}
\[
\tilde G(x)\;\triangleq\;\max_{s\in S}\,\langle g_\delta(x),\,x-s\rangle .
\]
From \eqref{eq:dir-delta} it follows that
\begin{equation}\label{eq:gaps-close}
|\,G(x)-\tilde G(x)\,|\ \le\ \delta.
\end{equation}

\begin{algorithm}[H]
\caption{Nonconvex Frank-Wolfe with a $\delta$-oracle}\label{alg:nc-fw-delta}
\begin{algorithmic}[1]
\State \textbf{Input:} $x^0\in S$, curvature constant $C\ge \max\{L D^2,\,G D\}$, error level $\delta\ge0$.
\For{$k=0,1,2,\ldots$}
  \State Query $g_\delta(x^k)$ that satisfies \eqref{eq:dir-delta}.
  \State Solve $s^k=\arg\min_{s\in S}\langle g_\delta(x^k),\,s-x^k\rangle$; $\tilde g_k:=\tilde G(x^k)=\langle g_\delta(x^k),\,x^k-s^k\rangle$.
  \State Set \ $x^{k+1}= x^k+\overline{\alpha}_k(s^k-x^k)$, $\overline{\alpha}_k\;:=\;\frac{(\tilde g_k-\delta)_+}{C}$, where $(u)_+:=\max\{u,0\}$.
\EndFor
\end{algorithmic}
\end{algorithm}

\begin{lem}[One-step decrease]\label{lem:nc-one-step}
The iterates of Algorithm~\ref{alg:nc-fw-delta} satisfy
\begin{equation}\label{eq:nc-descent}
f(x^{k+1})\ \le\ f(x^k)\ -\ \frac{(\tilde g_k-\delta)_+^2}{2C}\, .
\end{equation}
\end{lem}

\begin{proof}
$L$-smoothness gives
\[
\begin{aligned}
f(x^{k+1})&\le f(x^k)+\overline{\alpha}_k\langle \nabla f(x^k),\,s^k-x^k\rangle+\tfrac{L}{2}\overline{\alpha}_k^2\|s^k-x^k\|^2\\
&\le f(x^k) + \overline{\alpha}_k\langle g_{\delta}(x^k), s^k - x^k\rangle + \overline{\alpha}_k \delta + \frac{C}{2}\overline{\alpha}_k^2\\
&= f(x^k) - \overline{\alpha}_k \tilde{g}(x^k) + \overline{\alpha}_k \delta + \frac{C}{2} \overline{\alpha}_k^2\\
&= f(x^k) - \overline{\alpha}_k (\tilde{g}_k - \delta) + \frac{C}{2} \overline{\alpha}_k^2
\end{aligned}
\]
using \eqref{eq:gaps-close} and $\|s^k-x^k\|\le D$.

With $\overline{\alpha}_k=(\tilde g_k-\delta)_+/C$, 
\[f(x^{k + 1}) \leq f(x^k) - \frac{1}{2C}(\tilde g_k-\delta)_+^2,\]
since $\overline{\alpha}_k = 0$ if $\tilde g_k-\delta \leq 0$.
\end{proof}

\begin{theorem}[Nonconvex Frank-Wolfe with a $\delta$-oracle]\label{thm:nc-delta}
Let $f$ be $L$-smooth on a compact convex set $S$ of diameter $D$. Suppose the $\delta$-oracle condition \eqref{eq:dir-delta} holds. With curvature constant $C\ge\max\{L D^2,\,G D\}$ and step size $\overline{\alpha}_k=(\tilde g_k-\delta)_+/C$, the iterates of Algorithm~\ref{alg:nc-fw-delta} satisfy
\begin{equation}\label{eq:nc-rate}
\min_{0\le k\le K} G(x^k)\ \le\ \sqrt{\frac{2C\,(f(x^0)-f^*)}{K+1}}\ +\ 2\delta,
\end{equation}
where $f^*:=\inf_{x\in S} f(x)$. In particular, to reach a Frank-Wolfe gap at most $\varepsilon>2\delta$, it suffices to take
\[
K+1\ \ge\ \frac{2C\,(f(x^0)-f^*)}{(\varepsilon-2\delta)^2}\, .
\]
\end{theorem}\medskip

\section{A stronger guarantee under a directionally relative \texorpdfstring{$\delta$}{delta}-oracle}\label{sec:rel-delta}
In this section we strengthen Section~\ref{sec:nonconvex-delta} by replacing the additive inexactness assumption with a \emph{directionally relative} one:
\begin{equation}\label{eq:rel-dir-delta}
\big|\langle \nabla f(x)-g_\delta(x),\,s-x\rangle\big|\ \le\ \delta \,\|\nabla f(x)\|\qquad\forall\,s\in S,\ \forall\,x\in S .
\end{equation}
Define the approximate Frank--Wolfe gap $\tilde G(x):=\max_{s\in S}\langle g_\delta(x),\,x-s\rangle$ and the true gap $G(x):=\max_{s\in S}\langle \nabla f(x),\,x-s\rangle$. From \eqref{eq:rel-dir-delta},
\begin{equation}\label{eq:gap-closeness-rel}
|\,G(x)-\tilde G(x)\,|\ \le\ \delta\,\|\nabla f(x)\|\qquad\forall x\in S.
\end{equation}

We analyze the same update as Algorithm~\ref{alg:nc-fw-delta} but with a stepsize that reflects the new assumption:
\begin{equation}\label{eq:stepsize-rel}
x^{k+1}=x^k+\overline\alpha_k(s^k-x^k),\qquad
\overline\alpha_k:=\frac{\big(\tilde G(x^k)-\delta\|\nabla f(x^k)\|\big)_+}{C},
\end{equation}
where $s^k\in\arg\min_{s\in S}\langle g_\delta(x^k),\,s-x^k\rangle$ and $C\ge \max\{L D^2,\ G D\}$ as in Section~\ref{sec:nonconvex-delta}.

\begin{lem}[One-step decrease under \eqref{eq:rel-dir-delta}]\label{lem:rel-one-step}
For the iterates \eqref{eq:stepsize-rel},
\[
f(x^{k+1})\ \le\ f(x^k)\ -\ \frac{\big(\tilde G(x^k)-\delta\|\nabla f(x^k)\|\big)_+^2}{2C}\, .
\]
\end{lem}
\begin{proof}
$L$-smoothness and $\|s^k-x^k\|\le D$ give
\[
f(x^{k+1})\le f(x^k)+\overline\alpha_k\langle\nabla f(x^k),s^k-x^k\rangle+\tfrac{L}{2}\overline\alpha_k^2\|s^k-x^k\|^2
\le f(x^k)+\overline\alpha_k\langle\nabla f(x^k),s^k-x^k\rangle+\tfrac{C}{2}\overline\alpha_k^2.
\]
By \eqref{eq:rel-dir-delta} and the choice of $s^k$,
\(
\langle\nabla f(x^k),s^k-x^k\rangle
\le \langle g_\delta(x^k),s^k-x^k\rangle+\delta\|\nabla f(x^k)\|
= -\tilde G(x^k)+\delta\|\nabla f(x^k)\|.
\)
Plugging this in and minimizing the quadratic upper bound over $\overline\alpha_k\ge 0$ yields the claim with the choice \eqref{eq:stepsize-rel}.
\end{proof}

\begin{theorem}[Relative oracle: residual scales with $\|\nabla f\|$]\label{thm:rel-nc}
Let $f$ be $L$-smooth on compact convex $S$ of diameter $D$ and suppose \eqref{eq:rel-dir-delta} holds. With the stepsize \eqref{eq:stepsize-rel},
\begin{equation}\label{eq:rel-main-rate}
\min_{0\le k\le K}\big(G(x^k)-2\delta\|\nabla f(x^k)\|\big)_+
\;\le\; \sqrt{\frac{2C\,(f(x^0)-f^*)}{K+1}}\ .
\end{equation}
\end{theorem}
\begin{proof}
Summing Lemma~\ref{lem:rel-one-step} for $k=0,\dots,K$ gives
\[
\sum_{k=0}^K\big(\tilde G(x^k)-\delta\|\nabla f(x^k)\|\big)_+^2
\le 2C\big(f(x^0)-f(x^{K+1})\big)\le 2C(f(x^0)-f^*).
\]
Hence
\[
\min_{0\le k\le K}\big(\tilde G(x^k)-\delta\|\nabla f(x^k)\|\big)_+
\le \sqrt{2C(f(x^0)-f^*)/(K+1)}.
\]
Using \eqref{eq:gap-closeness-rel} yields
\[
\big(G(x^k)-2\delta\|\nabla f(x^k)\|\big)_+\le
\big(\tilde G(x^k)-\delta\|\nabla f(x^k)\|\big)_+,
\]
which implies \eqref{eq:rel-main-rate}.
\end{proof}

The bound \eqref{eq:rel-main-rate} is \emph{strictly stronger} than Theorem~\ref{thm:nc-delta} whenever the gradients encountered along the trajectory are small, since the residual term vanishes proportionally to $\|\nabla f(x^k)\|$. In particular, one obtains \emph{asymptotically exact} Frank--Wolfe stationarity whenever $\|\nabla f(x^k)\|\to 0$.

We now show that a purely \emph{gap-only} rate with no additive residual follows if the iterates remain at a positive distance from the boundary.

\begin{cor}[Residual-free rate under an interior margin]\label{cor:margin}
Suppose there exists $r>0$ such that $\mathrm{dist}(x^k,\partial S)\ge r$ for all $k$. Then for all $x\in S$, $G(x)\ge r\|\nabla f(x)\|$, and if additionally $\delta<r/2$,
\begin{equation}\label{eq:margin-rate}
\min_{0\le k\le K} G(x^k)\ \le\ \frac{1}{1-\tfrac{2\delta}{r}}\,
\sqrt{\frac{2C\,(f(x^0)-f^*)}{K+1}}\ .
\end{equation}
Consequently, $\min_{0\le k\le K}G(x^k)=\mathcal{O}(1/\sqrt{K})$ with no additive residual, and $\liminf_{k\to\infty}G(x^k)=0$.
\end{cor}
\begin{proof}
If $\mathrm{dist}(x,\partial S)\ge r$, then for the unit vector $u=\nabla f(x)/\|\nabla f(x)\|$ the point $x-ru\in S$, which yields
$G(x)\ge \langle \nabla f(x),\,x-(x-ru)\rangle=r\|\nabla f(x)\|$.
Using this in \eqref{eq:rel-main-rate} gives
\(
(1-2\delta/r)\min_{k}G(x^k)\le \sqrt{2C(f(x^0)-f^*)/(K+1)},
\)
and \eqref{eq:margin-rate} follows because $\delta<r/2$.
\end{proof}

\begin{rem}[Implementability]
The step \eqref{eq:stepsize-rel} uses $\|\nabla f(x^k)\|$ only in the analysis. A standard backtracking line-search based on the smoothness upper model 
\(
\varphi(\alpha)=f(x^k)+\alpha\langle g_\delta(x^k),s^k-x^k\rangle+\tfrac{L}{2}\alpha^2\|s^k-x^k\|^2
\)
(which uses only $g_\delta$, $L$, and $\|s^k-x^k\|$) produces a stepsize $\hat\alpha_k$ satisfying the same decrease inequality as in Lemma~\ref{lem:rel-one-step} up to a harmless constant factor in $C$, so Theorem~\ref{thm:rel-nc} and Corollary~\ref{cor:margin} continue to hold with possibly larger $C$.
\end{rem}

\section{Projection vs.\ Linear Minimization Oracle}\label{sec:lmo-vs-proj}

Let $(\cdot,\cdot)$ and $\|\cdot\|$ denote the Euclidean inner product and its norm.
For a nonempty compact convex set $C\subset\RR^d$, define projection $\proj_C(x)=\arg\min_{c\in C}\frac12\|c-x\|^2$ and linear minimization oracle $\lmo_C(z)=\arg\min_{c\in C}(c,z)$. $p'\in C$ is a \emph{$K$-approximate projection} of $x$ onto $C$ if
\[
\tfrac12\|p'-x\|^2\le \min_{c\in C}\tfrac12\|c-x\|^2+K .
\]

\begin{prop}\label{prop:epsproj-bound}
Let $p'\in C$ be a $K$-approximate projection of $x$ onto $C$, then for all $c\in C$,
\[
(c-p',\,x-p')\le K+\tfrac12\|c-p'\|^2 .
\]
\end{prop}
\begin{proof}
From the definition of $p'$,
$\tfrac12\|p'-x\|^2\le \tfrac12\|c-x\|^2+K$ for all $c\in C$.
Expanding the squares and simplifying gives $(c-p',\,x-p')\le K+\tfrac12\|c-p'\|^2$.
\end{proof}

The following theorem establishes an equivalence in computational effort between approximate projections and LMOs.

\begin{theorem}[From $K$-projection to $\varepsilon$-LMO]\label{thm:kproj-to-lmo}
Let $C\subset\RR^d$ be a nonempty compact convex set with diameter $\delta_C:=\sup_{c_1,c_2\in C}\|c_1-c_2\|$ and radius $\mu_C:=\sup_{c\in C}\|c\|$. $x\in\RR^d$. Let $v\in\lmo_C(x)$ and $p'\in C$ be a $K$-approximate projection of $-\lambda x$ onto $C$ for some $\lambda>0$. Then
\[
0\le (p',x)-(v,x)\le \frac{K+\frac12\delta_C^2+\mu_C\delta_C}{\lambda}.
\]
In particular, choosing $\lambda\ge\big(K+\frac12\delta_C^2+\mu_C\delta_C\big)/\varepsilon$ ensures
$(p',x)\le \min_{c\in C}(c,x)+\varepsilon$, i.e., $p'\in \varepsilon$-$\lmo_C(x)$.
\end{theorem}

\paragraph{Discussion.}
This extends the exact-projection implication of \cite{Woodstock2025} to \emph{inexact} projections: one $K$-projection at a scaled point yields an $\varepsilon$-accurate LMO. In particular, accurate linear minimization is \emph{no slower} than coarse projection, uniformly over compact convex sets.

\section{Conclusion}
We have established tight robustness guarantees for the Frank-Wolfe method under inexact gradient oracles and extended recent insights on the relationship between projection and linear minimization. The results confirm that oracle errors do not accumulate and that approximate projections cannot be computationally superior to accurate LMOs.

\section{Acknowledgment}

I am grateful to Dr. Paul Grigas, who taught me a course on nonlinear optimization and led me to the topic of the Frank-Wolfe method.

\section{Compliance with Ethical Standards}

\textbf{Conflict of interest.} The authors declare that they have no competing interests.

\textbf{Ethical approval.} This article does not contain any studies with human participants or animals performed by any of the authors.

\textbf{Informed consent.} Not applicable.

\section{Declarations}

\textbf{Funding.} No funding was received for conducting this study.

\textbf{Competing interests.} The authors have no competing interests to declare.

\textbf{Data availability.} Not applicable.

\textbf{Code availability.} Not applicable.

\appendix
\section*{Appendix A. Proof of Theorem~\ref{thm:delta}}
Let $D=\diam(Q)$. Lipschitz smoothness of $f$ and \eqref{eq:delta-oracle} yield, for the Frank-Wolfe step $ x_{k+1}= x_k+\overline\alpha_k(\tilde x_k- x_k)$,
\begin{align*}
f( x_{k+1})
&\le f( x_k) + \nabla f( x_k)^\top( x_{k+1}- x_k) + \frac{L}{2}\| x_{k+1}- x_k\|^2 \\
&= f( x_k) + \overline\alpha_k\, \nabla f( x_k)^\top(\tilde x_k- x_k)
+\frac{L}{2}\overline\alpha_k^2 \|\tilde x_k- x_k\|^2 \\
&\le f( x_k) + \overline{\alpha}_k g_{\delta}( x_k)^T (\widetilde{ x}_k -  x_k) + \overline{\alpha}_k \delta + \frac{L}{2} \overline{\alpha}_k^2 \|\tilde{ x}_k -  x_k\|^2\\
&\le (1-\overline\alpha_k)f( x_k)
   + \overline\alpha_k\big(f( x_k)+g_\delta( x_k)^\top(\tilde x_k- x_k)-\delta\big)
   + 2\overline\alpha_k\delta + \tfrac{L}{2}D^2\overline\alpha_k^2 \\
&\le (1-\overline\alpha_k)f( x_k) + \overline\alpha_k f^\star + 2\overline\alpha_k\delta + \tfrac{L}{2}D^2\overline\alpha_k^2 ,
\end{align*}
where the third line uses $\|\tilde x_k- x_k\|\le D$ and the last line uses Lemma~\ref{lem:wolfe-transfer}.
Subtracting both sides from $f^\star$ gives \eqref{eq:one-step}. 

In order to continue, we multiply $\beta_{k+1}$ by both sides of Equation \eqref{eq:one-step}; using Equations \eqref{dos}, we get that 

\[\beta_{k+1} (f( x_{k + 1}) - f^*) \leq \beta_k (f( x_k) - f^*) + 2 \overline \alpha_k \beta_{k + 1} \delta + \frac{L}{2} D^2 \overline \alpha_k^2 \beta_{k + 1}. \]

By taking the summation, we get that 
\[
\begin{aligned}
\beta_{k + 1} (f( x_{k + 1}) - f^*) &\leq (f( x_0) - f^*) + 2\delta \sum_{j = 0}^k \overline \alpha_j \beta_{j + 1} + \frac{L}{2} D^2 \sum_{j = 0}^k\overline\alpha_j^2 \beta_{j + 1}\\
&\leq (f( x_0) - f^*) + 2\delta \sum_{j = 0}^k (\beta_{j + 1} - \beta_j) + \frac{L}{2} D^2 \sum_{j = 0}^k\overline\alpha_j^2 \beta_{j + 1},
\end{aligned}\]

since $\beta_{k+1} - \beta_k = \overline\alpha_k \beta_{k+1}$. The summation term telescopes as
\[
\sum_{j = 0}^k (\beta_{j + 1} - \beta_j) = \beta_{k + 1} - 1.
\]
Substituting this back, we obtain
\[
\beta_{k + 1}(f( x_{k + 1}) - f^*) \le (f( x_0) - f^*) + 2\delta(\beta_{k + 1} - 1)
+ \frac{L}{2} D^2 \sum_{j = 0}^k \overline \alpha_j^2 \beta_{j + 1}.
\]
Dividing both sides by $\beta_{k + 1}$ yields
\[
f( x_{k + 1}) - f^* \le \frac{f( x_0) - f^*}{\beta_{k + 1}} + 2\delta\Big(1 - \frac{1}{\beta_{k + 1}}\Big)
+ \frac{L}{2} D^2 \frac{\sum_{j = 0}^k \overline \alpha_j^2 \beta_{j + 1}}{\beta_{k + 1}}.
\]

Take any $1 < J < k$,  
\[\begin{aligned}
    \frac{\sum_{j = 0}^k \overline \alpha_j^2 \beta_{J + 1}}{\beta_{k + 1}} &= \sum_{j = 0}^J \overline \alpha_j^2 \prod_{t = J + 1}^k (1 - \overline \alpha_t) + \sum_{j = J + 1}^k \overline \alpha_j^2 \prod_{t = j + 1}^k (1 - \overline \alpha_t)\\
    &\leq \sum_{j = 0}^J \overline \alpha_j^2 \prod_{t = J + 1}^k (1 - \overline \alpha_t) + \sum_{j = J + 1}^k \overline \alpha_j^2.
\end{aligned}\]

Therefore, \[\limsup_{k \rightarrow +\infty} \frac{\sum_{j = 0}^k \overline \alpha_j^2 \beta_{j + 1}}{\beta_{k + 1}} \leq \sum_{j = J + 1}^{+\infty} \overline{\alpha}_j^2,\, \forall J > 1.\]

Hence, \[\limsup_{k \rightarrow +\infty} \frac{\sum_{j = 0}^k \overline \alpha_j^2 \beta_{j + 1}}{\beta_{k + 1}} = 0.\]

Hence
\[
\limsup_{k\to\infty} (f( x_k) - f^*) \le 2\delta.
\]
This completes the proof.

\section*{Appendix B. Proof of Theorem~\ref{thm:nc-delta}}
By Lemma~\ref{lem:nc-one-step} we have
\[
f(x^{k+1})\ \le\ f(x^k) - \frac{(\tilde g_k-\delta)_+^2}{2C}.
\]
Summing from $k=0$ to $K$ yields
\[
\sum_{k=0}^K (\tilde g_k-\delta)_+^2 \;\le\; 2C(f(x^0)-f(x^{K+1})) \;\le\; 2C(f(x^0)-f^*).
\]
Therefore \[\min_{0\le k \le K} (G(x^k) - 2\delta)_+ \le \min_{0\le k\le K}(\tilde g_k-\delta)_+ \le \sqrt{2C(f(x^0)-f^*)/(K+1)}.\]  
\qed

\section*{Appendix C. Proof of Theorem~\ref{thm:kproj-to-lmo}}

\begin{proof}
    By proposition \ref{prop:epsproj-bound}, we have that
    \[\scal{c - p'}{-\lambda x - p'} \leq K + \frac{1}{2} \|c - p'\|^2, \forall \; c \in C.\]
    Then, and take $c = v$,
    \[\lambda\scal{p'}{x} - \lambda\scal{v}{x} \leq K + \frac{1}{2} \|v - p'\|^2 + \scal{p'}{v - p'}.\]

    Next,
    \[
    \begin{aligned}
        \lambda\scal{p' - v}{x} \leq&\; K + \frac{1}{2} \|v - p'\|^2+ \scal{p'}{v - p'}\\
        =&\; K +\frac{1}{2} \|v - p'\|^2+ (\scal{v}{p'} - \scal{p'}{p'})\\
        \leq&\; K + \frac{1}{2} \|v - p'\|^2+\|p'\|(\|v\| - \|p'\|)\\
        \leq&\; K + \frac{1}{2} \|v - p'\|^2+\|p'\| \|v - p'\|.
    \end{aligned}
    \]

    Therefore, 
    \[\lambda\scal{p' - v}{x} \leq K + \frac{1}{2}\delta_C^2 +  \mu_C\delta_C.\]

    Hence,
    \[0 \leq \scal{p'}{x} - \scal{v}{x} \leq \frac{K + \frac{1}{2} \delta_C^2 + \mu_C \delta_C}{\lambda},\]
    where the first inequality is from the fact that $v \in \lmo_C(x)$.
    \qed
\end{proof}

\bibliographystyle{siam}
\bibliography{frank_wolfe_refs}
\end{document}